\documentclass{article}
\usepackage{amssymb,amsmath}
\usepackage{amsthm}

\usepackage{graphicx}

\def\qed{\hfill {\hbox{${\vcenter{\vbox{               
   \hrule height 0.4pt\hbox{\vrule width 0.4pt height 6pt
   \kern5pt\vrule width 0.4pt}\hrule height 0.4pt}}}$}}}

\newtheorem*{theorem*}{Theorem}
\newtheorem{theorem}{Theorem}
\newtheorem{definition}{Definition}
\newtheorem*{definition*}{Definition}
\newtheorem{lemma}[theorem]{Lemma}

\newtheorem{example}{Example}

\theoremstyle{plain}

\newcommand{\rst}[1]{\ensuremath{{\mathbin\vert}%
\raise-.5ex\hbox{$#1$}}}

\date{}

\title{\Large \textbf{A Generalization of Lie's Theorem}}

\author{Johanna Hennig} 

\begin{document}


\maketitle

\begin{abstract}
We prove that in a locally finite dimensional Lie algebra $L$, any maximal, locally solvable subalgebra is the stabilizer of a maximal, generalized flag in an integrable, faithful module over $L$.
\end{abstract}

\bigskip


\bigskip


\section{\large \textbf{Introduction}}

In \cite{DP1}, Dimitrov and Penkov extended the classical Lie's theorem to the infinite dimensional case. Let $V$ be a vector space over an algebraically closed field $k$ of zero characteristic. Let $fgl(V)$ denote the Lie algebra of finite rank linear transformations on $V$. Dimitrov and Penkov proved that every maximal, locally solvable subalgebra of $fgl(V)$ is the stabilizer of a maximal, generalized flag in $V$. Moreover, there is a 1-1 correspondence between maximal, closed generalized flags in $V$ and maximal, locally solvable subalgebras of $fgl(V)$.

\bigskip

\noindent In this paper we use Model Theory to prove the following:

\smallskip

\begin{theorem*}\textup{Let $L$ be a locally finite dimensional Lie algebra over an algebraically closed field $k$ of zero characteristic. Let $V$ be an integrable, faithful module over $L$. Then every maximal, locally solvable subalgebra of $L$ is the stabilizer of a maximal, generalized flag in $V$.}
\end{theorem*}

\smallskip

\noindent In \cite{DP1} it is shown that we cannot expect a 1-1 correspondence between maximal, locally solvable subalgebras and maximal generalized flags in this generality. Moreover, the work of Dan-Cohen in \cite{EDC} shows that we cannot even expect a 1-1 correspondence between maximal, locally solvable subalgebras and maximal, \textit{closed} generalized flags in an arbitrary locally finite Lie algebra.

\section{\large \textbf{Preliminaries}}

\noindent A Lie algebra $L$ is \textit{locally finite dimensional}, or simply \textit{locally finite}, if every finitely generated subalgebra is finite dimensional. An $L$-module $V$ is \textit{integrable} if for every finite subset $\{v_1, \dots , v_n \}$ of $V$ and for every finitely generated subalgebra $L_\alpha$ of $L$, $\{v_1, \dots , v_n \}$ is contained in a finite dimensional $L_\alpha$-submodule. Equivalently, $V$ is an integrable module if and only if the Lie algebra $\Tilde{L} = L + V$ is locally finite.

\bigskip

\noindent In a finite dimensional vector space $V$, a \textit{flag} is a chain of subspaces $(0) \subsetneq V_1 \subsetneq V_2 \subsetneq \dots \subsetneq V_n = V$, where dim $V_i = i$. In an infinite dimensional vector space, this definition will no longer suffice--we will instead use \textit{generalized flags}, which were introduced in \cite{DP2}.

\begin{definition*} 
\textup{Let $V$ be a vector space. A \textit{generalized flag} $\mathbf{F}$ in $V$ is a set of subspaces totally ordered by inclusion such that}
\begin{list}{}{}
\item[i)] \textup{each subspace $S \in \mathbf{F}$ has an immediate predecessor or an immediate successor.}
\item[ii)] \textup{for every nonzero $v \in V$, there is a pair $F'_v, F_v'' \in \mathbf{F}$ such that $F_v''$ is the immediate successor of $F_v'$ and $v \in F_v'' -F_v'$. }
\end{list}
\end{definition*}

\noindent We may write $\mathbf{F} = \{F'_\alpha, F''_\alpha \}_{\alpha \in A}$ where $F'_\alpha$ is the immediate predecessor of $F''_\alpha$ and $A$ is linearly ordered by $\alpha \prec \beta$ if and only if $F'_{\alpha} \subset F'_{\beta}$. A generalized flag $\mathbf{F} = \{F'_\alpha, F''_\alpha \}_{\alpha \in A}$ is \textit{maximal} if and only if dim $F''_\alpha / F'_\alpha = 1$ for all $\alpha \in A$.

\bigskip

\noindent Let $X$ be a set. For the definition and properties of ultrafilters on the set $X$, see \cite{H} or \cite{M}. We will use the following theorem of Malcev, found in \cite{M}:

\begin{theorem*}{(Malcev)}
\textup{Every algebraic system embeds into an ultraproduct of its finitely generated subsystems.}
\end{theorem*}

\noindent The following lemma gives examples of maximal, generalized flags using an ultrafilter $\mathcal{F}$.

\begin{lemma} \label{flags}
\textup{Let $X$ be a set, and let $\{V_\alpha \}_{\alpha \in X}$ be a family of finite dimensional vector spaces. Let $F_\alpha$ be a flag in $V_\alpha$ and let $\mathcal{F}$ be an ultrafilter on $X$. Then $ \prod F_\alpha / \mathcal{F}$ is a maximal, generalized flag in the vector space $\prod V_\alpha / \mathcal{F}$.}
\end{lemma}

\begin{proof}
For each $\alpha \in X$, $F_\alpha$ is a flag in $V_\alpha$, i.e. a chain of subspaces $(0) \subsetneq V^{\alpha}_1 \subsetneq \dots \subsetneq V^{\alpha}_{n_\alpha}$ where $n_\alpha =$ dim $V_\alpha$. Each element $f \in \prod F_\alpha / \mathcal{F}$ is a function on $X$: for each $\alpha \in X$, $f(\alpha) = V^{\alpha}_i$, a subspace which appears in the flag $F_\alpha$, with the usual equivalence relation given by the ultrafilter. Thus, each element of $\prod F_\alpha / \mathcal{F}$ is a subspace of $\prod V_\alpha / \mathcal{F}$.

\bigskip

\noindent We can also identify each element $f \in \prod F_\alpha / \mathcal{F}$ with a function $q: X \rightarrow \mathbb{Z}$ such that for all $\alpha \in X$, $q(\alpha) \in \{ 0, 1, \dots , n_\alpha \}$, where $q(\alpha) = i \Leftrightarrow f(\alpha) = V^{\alpha}_i$. We shall adopt this viewpoint of elements of $\prod F_\alpha / \mathcal{F}$ for the remainder of the proof.

\bigskip

\noindent We totally order $\prod F_\alpha / \mathcal{F}$ by inclusion: for $p$, $q \in\prod F_\alpha / \mathcal{F}$, we have 

\begin{eqnarray*}
A_1 & = & \{\alpha | p(\alpha) > q(\alpha)\} \\
A_2 & = & \{\alpha | p(\alpha) = q(\alpha)\} \\
A_3 & = & \{\alpha | p(\alpha) < q(\alpha)\} \\
\end{eqnarray*}

\noindent Since $\mathcal{F}$ is an ultrafilter, only one of the $A_i$'s are in $\mathcal{F}$: if $A_1 \in \mathcal{F}$, then we say $p > q$, if $A_2 \in \mathcal{F}$, then $p = q$, if $A_3 \in \mathcal{F}$, then $p < q$. Note that $p \leq q$ if and only if $p$ is a subset of $q$. Hence $\prod F_\alpha / \mathcal{F}$ is a set of subspaces in $\prod V_\alpha / \mathcal{F}$ totally ordered by inclusion. 

\bigskip

\noindent It is clear that the subspace $(0)$, which is identified with the function $q_{min} (\alpha) = 0$ $\forall \alpha$, is minimal with respect to this ordering, and $\prod V_\alpha / \mathcal{F}$, which is identified with the function $q_{max}(\alpha) = n_\alpha$ $\forall \alpha$, is maximal. Every other $q \in \prod F_\alpha / \mathcal{F}$ actually has both an immediate predecessor and successor: for each $q$, we define the successor as $(q+1)(\alpha) = q(\alpha) +1$ for all $\alpha$ such that $q(\alpha) \ne n_\alpha$ and the predecessor as $(q-1)(\alpha) = q(\alpha) -1$ for all $\alpha$ such that $q(\alpha) \ne 0$. Then $(q - 1) < q < (q+1)$ and we claim that dim $(q+1)/q = 1$ for each $q \ne q_{max}$: the set $A = \{ \alpha \in X | q(\alpha) < n_\alpha \}$ must be in $\mathcal{F}$, since $q \ne q_{max}$, so for each $\alpha \in A$ we can choose $w_\alpha \in (q +1)(\alpha) - (q)(\alpha)$ and define $w \in \prod V_\alpha / \mathcal{F}$ by $w(\alpha) = w_\alpha$ for $\alpha \in A$ and $w(\alpha) = 0$ else. Then $w \in (q+1)-q$, and since dim $(q+1)(\alpha) / q(\alpha) = 1$ for $\alpha \in A$ and $\{\alpha \in X | q(\alpha) + kw_\alpha =  (q+1)(\alpha) \} \supset A \in \mathcal{F}$, we have that $(q+1) = q + kw$. Hence dim $(q+1)/q = 1$, which implies that $\prod F_\alpha / \mathcal{F}$ is a maximal chain of subspaces in $\prod V_\alpha / \mathcal{F}$.

\bigskip

\noindent Lastly, let $v \in \prod V_\alpha / \mathcal{F}$ be given such that $v \ne 0$, that is, $A = \{\alpha \in X | v(\alpha) \ne 0 \} \in \mathcal{F}$. Then for each $\alpha \in A$, we can choose some $i_\alpha \in \{0,1, \dots, n_\alpha \}$ such that $v(\alpha) \in V^{\alpha}_{i_{\alpha}} - V^{\alpha}_{i_{\alpha} - 1}$. Define $q_v \in \prod F_\alpha / \mathcal{F}$ by $q_v(\alpha) = i_{\alpha}$ for $\alpha \in A$ and $q_v(\alpha) = 0$ else. Then $v \in (q_v) - (q_v - 1)$. Hence, $\prod F_\alpha / \mathcal{F}$ is a maximal generalized flag in $\prod V_\alpha / \mathcal{F}$. \end{proof}

\bigskip

\noindent A Lie algebra $L$ is \textit{locally solvable} if every finitely generated subalgebra is solvable.

\begin{lemma} \label{stabilizer}
\textup{Suppose $L$ is a locally finite Lie algebra and $V$ is an integrable, faithful module over $L$. If $L$ stabilizes a maximal, generalized flag in $V$, then $L$ is locally solvable.}
\end{lemma}

\begin{proof}
Let $X$ be the set of nonempty finite subsets of $L$, and let $L_\alpha$ be the finite dimensional Lie algebra generated by $\alpha \in X$. To show $L$ is locally solvable, it suffices to show that each $L_\alpha$ is a solvable Lie algebra, which we do by constructing a finite dimensional faithful $L_\alpha-$module and applying the classical Lie's theorem.

\bigskip

\noindent Fix $\alpha \in X$, a finite subset of $L$. We claim that we can find a faithful, finite dimensional $L_\alpha$--submodule of $V$. Choose some nonzero $v_1 \in V$, which by assumption is contained in a finite dimensional $L_\alpha$--submodule $V_1$. If $L_\alpha$ acts faithfully on $V_1$, we are done. Otherwise, there is some nonzero $x \in L_\alpha$ such that $x \cdot V_1 = 0$, that is, $x \in K_1$, the kernal of the action of $L_\alpha$ on $V_1$. Since $0 \ne x \in L$ and $L$ acts faithfully on $V$, there is some $v_2 \in V$ such that $x \cdot v_2 \ne 0$. Let $V_2$ be the finite dimensional $L_\alpha$--submodule generated by $V_1$ and $v_2$. If $L_\alpha$ acts faithfully on $V_2$, we are done. If not, then the kernal of the action of $L_\alpha$, denoted $K_2$, is nontrivial. However, since $x \cdot v_2 \ne 0$, we have that $x \cdot V_2 \ne 0$. Thus $x \in K_1 - K_2$ and dim $K_2 <$ dim $K_1 \le$ dim $L_\alpha$. We can repeat this process, creating a larger finite dimensional $L_\alpha$--module with a strictly smaller kernal. Since $L_\alpha$ is finite dimensional, this will eventually terminate with a finite dimensional faithful $L_\alpha$-module, which we call $V_\alpha$.

\bigskip

\noindent Let $\mathbf{F}$ denote the maximal generalized flag stabilized by $L$, and hence also by $L_\alpha$. It is straightforward to check that the set of distinct subspaces from $\{F \cap V_{\alpha} | F \in \mathbf{F}\}$ is a flag in $V_{\alpha}$, which is invariant under $L_{\alpha}$. Since $L_\alpha$ stabilizes a flag in the finite dimensional, faithful module $V_\alpha$, it follows that $L_\alpha$ is solvable.
\end{proof}

\section{\large \textbf{Proof of the Main Theorem}}

We now prove the main result of this paper.

\begin{theorem*}\textup{Let $L$ be a locally finite dimensional Lie algebra over an algebraically closed field $k$ of zero characteristic. Let $V$ be an integrable, faithful module over $L$. Then every maximal, locally solvable subalgebra of $L$ is the stabilizer of a maximal, generalized flag in $V$.}
\end{theorem*}

\begin{proof}
Let $B$ be a maximal, locally solvable subalgebra of $L$, hence $V$ is an integrable module over $B$. Let $X$ denote the set of pairs $\alpha = (\alpha', \alpha'')$, where $\alpha'$ is a nonempty finite subset of $B$ and $\alpha''$ is a nonempty finite subset of $V$. Let $B_\alpha$ denote the finite dimensional subalgebra of $B$ generated by $\alpha'$, and let $V_\alpha$ denote the finite dimensional $B_\alpha$-submodule of $V$ generated by $\alpha''$. 

\bigskip

\noindent Let $X_\alpha = \{ \beta = (\beta', \beta'') \in X | \alpha' \subseteq \beta', \alpha'' \subseteq \beta''\}$. Since finite intersections of elements from the set $\{X_\alpha \}_{\alpha \in X}$ are nonempty,  $\{X_\alpha \}$ embeds into an ultrafilter, $\mathcal{F}$. Thus, by Malcev's theorem, $B$ embeds into the Lie algebra $\Tilde{B} = \prod B_\alpha / \mathcal{F}$, $V$ embeds into the vector space $\Tilde{V} = \prod V_\alpha / \mathcal{F}$, and since $B_\alpha$ acts on $V_\alpha$, we have that $\Tilde{B}$ acts on $\Tilde{V}$ via: $(b \cdot v)(\alpha) = b(\alpha) \cdot v(\alpha)$.

\medskip
\begin{eqnarray*}
B &\overset{\psi}{\hookrightarrow} & \Tilde{B}  = \prod B_\alpha / \mathcal{F}\\
\downarrow & & \downarrow\\
V & \overset{\phi}{\hookrightarrow}& \Tilde{V}  = \prod V_\alpha / \mathcal{F}\\
\end{eqnarray*}
\medskip

\noindent We claim that this embedding respects the action of $B$ on $V$, i.e. $\phi(V)$ is an $\psi(B)$-submodule of $\Tilde{V}$. We need to show for $x \in B$ and $v \in V$, $\phi(x \cdot v) = \psi(x) \cdot \phi(v)$ in $\Tilde{V}$, which is equivalent to the set $A = \{ \alpha \in X | \phi(x \cdot v)(\alpha) = \psi(x)(\alpha) \cdot \phi(v)(\alpha)\} \in \mathcal{F}$. Define $\beta = (\beta', \beta '') \in X$ by $\beta' = \{x\}$ and $\beta'' = \{ v\}$. Then certainly $x \in B_\beta$, $v \in V_\beta$, and $x \cdot v \in V_\beta$, so $\phi(x \cdot v)(\beta) = x \cdot v = \psi(x)(\beta) \cdot \phi(v)(\beta)$. Thus $A \supseteq X_\beta \in \mathcal{F}$, so $A \in \mathcal{F}$ as well. Hence, we can think of $V$ as an $B$-submodule of $\Tilde{V}$.

\bigskip

\noindent Since $B$ is locally solvable, each $B_\alpha$ is a finite dimensional solvable Lie algebra which acts on $V_\alpha$. Hence by Lie's theorem, $B_\alpha$ stabilizes a flag $F_\alpha$ in $V_\alpha$. Since each $F_\alpha$ is stable under $B_\alpha$, it follows that $\Tilde{F} = \prod F_\alpha / \mathcal{F}$ is stable under $\Tilde{B}  = \prod B_\alpha / \mathcal{F}$. By lemma \ref{flags} the ultraproduct $\Tilde{F} = \prod F_\alpha / \mathcal{F}$ is a maximal generalized flag in $\Tilde{V}= \prod V_\alpha / \mathcal{F}$ which is stabilized by $\Tilde{B}  = \prod B_\alpha / \mathcal{F}$, and hence is stabilized by $B$ as well.

\bigskip

\noindent To obtain a maximal generalized flag in $V$, consider the set $\Tilde{F} \cap V = \{q \cap V \}_{q \in \Tilde{F}}$ (Here we are using the copy of $V$ embedded in $\Tilde{V}$ by $\phi$). This is a chain of $B$-submodules of $V$ which may not be a generalized flag--however, we can generate a maximal generalized flag by defining for each $0 \ne v \in V$, $F''_v = q_v \cap V$ and $F'_v = (q_v-1) \cap V$ using the definition for $q_v$ as in the proof of lemma \ref{flags}. Then $\mathbf{F} = \{F'_v, F''_v \}_{v \in V}$ is a generalized flag which is maximal, since dim $F''_v / F'_v = 1$ for each $v \in V$. Since each $q \in \Tilde{F}$ is stabilized by $B$, so is each $q \cap V$ and hence $F''_v$ and $F'_v$ for each $v \in V$. Hence, $\mathbf{F}$ is a maximal generalized flag in $V$ which is stabilized by $B$.

\bigskip

\noindent Thus, we have that $B \subseteq St_\mathbf{F}$, the stabilizer of $\mathbf{F}$ in $L$. By lemma \ref{stabilizer}, $St_\mathbf{F}$ is a locally solvable subalgebra of $L$. Since $B$ is a maximal, locally solvable subalgebra, it must be that $B = St_\mathbf{F}$. Hence, $B$ is the stabilizer of the maximal, generalized flag $\mathbf{F}$.\end{proof}

\bigskip

\section{\large \textbf{Appendix: Ultrafilters and Ultraproducts}}\label{U}

\noindent The results in this section are standard, and we refer the reader to \cite{H} or \cite{M} for more details.

\subsection{\large \textbf{Ultrafilters}}

\begin{definition} 
\textup{A \textit{filter} over a set $X \ne \emptyset$ is a collection of subsets $\mathcal{F}$  such that }
\begin{list}{}{}
\item[i)] $\emptyset \notin \mathcal{F}$
\item[ii)] If $S_1$, $S_2 \in \mathcal{F}$, then $S_1 \cap S_2 \in \mathcal{F}$
\item[iii)] If $S_1 \in \mathcal{F}$ and $S_1 \subset T$, then $T \in \mathcal{F}$
\end{list}
\end{definition}
 
\begin{example} 

\textup{If $X$ is any infinite set, the collection of cofinite sets forms a filter on $X$.}  
\end{example}

\noindent There is a partial ordering on filters: $D_1 \leq D_2$ if $\forall S \in D_1$, $S \in D_2$. An \textit{ultrafilter} is a maximal filter with respect to this ordering. By Zorn's lemma, every filter may be embedded into an ultrafilter.

\begin{lemma} \label{embedding lemma}
\textup{A system $\mathcal{S}$ of subsets of $X$ embeds into a filter (and hence ultrafilter) if and only if all finite intersections of elements from $\mathcal{S}$ are nonempty}. 

\end{lemma}

\begin{proof}
\noindent First, suppose that $\mathcal{S}$ is contained in a filter $\mathcal{F}$. Then if $A_1, \dots, A_n$ is a finite collection of elements from $\mathcal{S}$, each of the $A_i$ is in $\mathcal{F}$, and so $A_1 \cap \cdots \cap A_n \in \mathcal{F}$ and hence must be nonempty.

\noindent Now, suppose $\mathcal{S}$ satisfies the finite intersection property, and let 
\begin{eqnarray*}
\mathcal{F} = \{ D \subseteq X | A_1 \cap \cdots \cap A_n \subseteq D, A_1, \dots, A_n \in \mathcal{S}\}
\end{eqnarray*}

\noindent In other words, the set $\mathcal{F}$ is obtained by including all supersets of all intersections of finitely many elements from $\mathcal{S}$. Clearly, $\mathcal{S}$ is contained in such a set. We wish to show $\mathcal{F}$ is a filter. 

\noindent First, $\emptyset \notin \mathcal{F}$ since every finite intersection of elements from $\mathcal{S}$ is nonempty.  If $A$, $B \in \mathcal{F}$, i.e. $A_1 \cap \cdots \cap A_n \subseteq A$ for some $A_i \in \mathcal{S}$, and $B_1 \cap \cdots \cap B_m \subseteq B$ for some $B_j$, then $A_1 \cap \cdots \cap A_n \cap B_1 \cap \cdots \cap B_m \subseteq A \cap B$, and hence $A \cap B \in \mathcal {F}$, so $\mathcal{F}$ is closed under finite intersections. Lastly, if $A \in \mathcal{F}$ and $A \subseteq B$, then $A_1 \cap \cdots \cap A_n \subseteq A$ for some $A_i \in \mathcal{S}$ and so $A_1 \cap \cdots \cap A_n \subseteq A \subseteq B$ as well, implying $B \in \mathcal{F}$ and so $\mathcal{F}$ is closed under supersets and hence forms a filter.
\end{proof}

\begin{lemma} \label{ultrafilter lemma}
\textup{A filter $\mathcal{F}$ on $X$ is an ultrafilter if and only if given any subset $T \subset X$, $T \in \mathcal{F}$ or $X - T \in \mathcal{F}$}. 
\end{lemma}

\begin{proof}
\noindent First, if $\mathcal{F}$ is a filter with this property, then $\mathcal{F}$ must be maximal: Suppose $\mathcal{F}$ is strictly contained in another filter $\mathcal{F'}$, i.e. there is some subset $T \in \mathcal{F'} - \mathcal{F}$. Since $T$ is not in $\mathcal{F}$, $X - T \in \mathcal{F}$ by assumption and hence $X - T \in \mathcal{F'}$ as well. But then both $T$ and $X-T$ are in $\mathcal{F'}$, implying that their intersection $T \cap (X - T) = \emptyset \in \mathcal{F'}$, a contradiction.

\noindent Now, suppose $\mathcal{F}$ is an ultrafilter and $T \subset X$ such that $T$ is not in $\mathcal{F}$. We wish to show $X - T \in \mathcal{F}$. If for all $S \in \mathcal{F}$, $S \cap T \neq \emptyset$, then by Lemma \ref{embedding lemma} $T$ and $\mathcal{F}$ embed into a strictly larger filter, which contradicts the maximality of $\mathcal{F}$. Hence, we may choose some $S_1 \in \mathcal{F}$ such that $T \cap S_1 = \emptyset$. If $X-T$ is also not in $\mathcal{F}$, then there is some $S_2 \in \mathcal{F}$ such that $(X-T) \cap S_2 = \emptyset$. Since $S_1$ and $S_2$ are in $\mathcal{F}$, $S_1 \cap S_2 \ne \emptyset$. Then we have:

\begin{eqnarray*}
 \emptyset & \ne & S_1 \cap S_2 = X \cap (S_1 \cap S_2)  = (T \cup (X-T)) \cap (S_1 \cap S_2) \\
   && = (T \cap S_1 \cap S_2) \cup (X-T \cap S_1 \cap S_2) = \emptyset\\
\end{eqnarray*}

\noindent This is a contradiction; hence, we must have that $X - T \in \mathcal{F}$. 
\end{proof}

\subsection{\large \textbf{Ultraproducts}}

\noindent Suppose $\{A_\alpha | \alpha \in X\}$ is a collection of algebraic structures of the same type (i.e. each $A_\alpha$ is a group, field, Lie algebra, vector space, etc.) Then we can define the Cartesian product $\prod A_\alpha$ which consists of functions $f: X \rightarrow \prod A_\alpha$ where $f(\alpha) \in A_\alpha$ for all $\alpha \in X$. $\prod A_\alpha$ is the same type of algebraic structure as each of its constituents, where each operation is defined pointwise.

\medskip

\noindent Now let $\mathcal{F}$ be an ultrafilter on $X$. Define an equivalence relation on $\prod A_\alpha$ by $f \sim g$ if $\{ \alpha \in X | f(\alpha) = g(\alpha) \} \in \mathcal{F}$. The \textit{ultraproduct} $\prod A_\alpha / \mathcal{F}$ is the set of equivalence classes, and one may check that this is again the same kind of algebraic structure as each of the $A_\alpha$.  

\begin{example}
\textup{Let $X \subset \mathbb{Z}$ be the set of all primes, and let $\mathcal{F}$ be an ultrafilter on $X$ containing all cofinite sets.  Let $F_p = \mathbb{Z}/p\mathbb{Z}$, and form the ultraproduct $\prod F_p / \mathcal{F}$. This is again a field which, in fact, has characteristic zero}.
\end{example}

\begin{lemma}\label{totally order lemma}
\textup{Let $\mathcal{F}$ be an ultrafilter on $X$. If $A_1, \dots, A_n$ are pairwise disjoint subsets such that $\bigcup A_i \in \mathcal{F}$, then exactly one of the $A_i$ is in $\mathcal{F}$}.
\end{lemma}

\begin{proof}
\noindent We use Lemma \ref{ultrafilter lemma}: since $\bigcup A_i \in \mathcal{F}$, we have $X - \bigcup A_i = \bigcap (X - A_i) \notin \mathcal{F}$, and so $X-A_i \notin \mathcal {F}$ for some $i$ since $\mathcal{F}$ is closed under finite intersections, and so again by Lemma \ref{ultrafilter lemma}, $A_i \in \mathcal{F}$. If more than one $A_i \in \mathcal{F}$, say $A_i$ and $A_j$ are both in $\mathcal{F}$ for $i \ne j$, then since they are disjoint, $\emptyset = A_i \cap A_j \in \mathcal{F}$ which is a contradiction. Hence, one and only one of the $A_i$ is in $\mathcal{F}$.
\end{proof}

\begin{example} \label{total order example}
\textup{Let $\mathcal{F}$ be an ultrafilter on $\mathbb{N}$ containing all cofinite sets, and let $\mathbb{R}^\mathbb{N}$ be the set of sequences in $\mathbb{R}$. Then we can form the ultraproduct $\mathbb{R}^\mathbb{N} / \mathcal{F}$, which contains a copy of $\mathbb{R}$ via the map $ a \in \mathbb{R} \mapsto (a, a, a, \dots) \in \mathbb{R}^\mathbb{N} / \mathcal{F}$. Let $(a_n)$ and $(b_n)$ be two sequences and consider the sets}:
\begin{eqnarray*}
A_1 & = & \{n | a_n > b_n\} \\
A_2 & = & \{n | a_n = b_n\} \\
A_3 & = & \{n | a_n < b_n\} \\
\end{eqnarray*}

\noindent \textup{Since $A_1 \cup A_2 \cup A_3 = \mathbb{N} \in \mathcal{F}$, by Lemma \ref{totally order lemma} we have that exactly one of the $A_i$ is in $\mathcal{F}$ and from this we can define a total ordering on $\mathbb{R}^\mathbb{N} / \mathcal{F}$: if $A_1 \in \mathcal{F}$, we say $(a_n) > (b_n)$, if $A_2 \in \mathcal{F}$, then $(a_n) = (b_n)$, and if $A_3 \in \mathcal{F}$, $(a_n) < (b_n)$.}

\medskip

\noindent \textup{Now consider the element $\epsilon = (1, \frac{1}{2}, \frac{1}{3}, \dots )$, i.e. $\epsilon(n) = \frac{1}{n}$ $\forall n \in \mathbb{N}$. In the ultraproduct, $\epsilon < \frac{1}{n} = (\frac{1}{n}, \frac{1}{n}, \frac{1}{n}, \dots)$ for all $n \in \mathbb{N}$ since $\{n | \epsilon(n) < \frac{1}{n}\}$ is a cofinite set and hence in $\mathcal{F}$; however, unlike in the standard real numbers, $\epsilon \ne 0$.}
\end{example} 

\noindent The next theorem is a less precise wording of a theorem in \cite{M} (theorem 2, section 8.3). We state it in this way to elucidate its meaning and its applications.

\begin{theorem} \label{ultraproduct theorem}
\textup{Every algebraic system embeds into an ultraproduct of its finitely generated subsystems.}
\end{theorem}

\begin{proof}
\noindent Let $A$ be an algebraic system, let $X$ be the set of all nonempty finite subsets of $A$, i.e. $X = \{ \alpha \subset A | \alpha \ne \emptyset, | \alpha | < \infty \}$. Clearly, $X$ is an index set for the finitely generated subsystems of $A$. For $\alpha \in X$, define $X_{\alpha} = \{ \beta \in X | \alpha \subseteq \beta\}$, which will be nonempty since $\alpha \in X_{\alpha}$. Then $\{ X_{\alpha}\}_{\alpha \in X}$ is a system of subsets of $X$ such that any finite intersection is nonempty: indeed, for $\alpha, \beta \in X$ we have:
\begin{eqnarray*}
X_\alpha \cap X_\beta = \{ \gamma \in X | \alpha \subseteq \gamma, \beta \subseteq \gamma\} = \{ \gamma \in X | \alpha \cup \beta \subseteq \gamma\} = X_{\alpha \cup \beta} \ne \emptyset
\end{eqnarray*}
Hence, by Lemma \ref{embedding lemma}, we are able to embed $\{ X_{\alpha}\}_{\alpha \in X}$ into an ultrafilter, $\mathcal{F}$. 

\noindent For $\alpha \in X$, let $A_\alpha$ be the subsystem of $A$ generated by $\alpha$ (i.e. the smallest subsystem of $A$ containing the finite set $\alpha$). Let $\Tilde{A}$ denote the ultraproduct $\prod A_\alpha / \mathcal{F}$. We wish to show that $A$ embeds into $\Tilde{A}$. 

\noindent Choose $d \in \Tilde{A}$, i.e. $d: X \rightarrow \prod A_\alpha$ where $d(\alpha) \in A_\alpha$ (the choice of $d$ does not matter). Define the map $\psi: A \rightarrow \Tilde{A}$ by:

\medskip

$\mbox{~~~~~~~~~~} x \in A \mapsto \psi_x \in \Tilde{A} = \prod A_\alpha / \mathcal{F} \mbox{~~~~where~~~~}\psi_x(\alpha) =	\left\{ \begin{array}{rcl} x & \mbox{if} & x \in A_\alpha \\ d(\alpha) & \mbox{if} & x \notin A_\alpha
\end{array}\right.$

\noindent Suppose we have an operation $\ast$ defined on $A$ and hence on each $A_\alpha$ and on the ultraproduct $\Tilde{A}$. Then we show that $\psi(x \ast y) = \psi_{x \ast y} = \psi_x \ast \psi_y = \psi(x) \ast \psi(y)$ in $\Tilde{A}$, i.e. the set $\{ \alpha \in X |  \psi_{x \ast y} (\alpha) = \psi_x (\alpha) \ast \psi_y (\alpha)\}$ is in $\mathcal{F}$.

\begin{eqnarray*}
\{ \alpha \in X |  \psi_{x \ast y} (\alpha) = \psi_x (\alpha) \ast \psi_y (\alpha)\} & \supseteq & \{ \alpha \in X |  \psi_{x \ast y} (\alpha) = x \ast y = \psi_x (\alpha) \ast \psi_y (\alpha)\} \\
& \supseteq & \{ \alpha \in X | \psi_x (\alpha) = x\} \cap \{ \alpha \in X |  \psi_y (\alpha) = y\} \\
& \supseteq & X_{\alpha} \cap X_{\beta} \in \mathcal{F}. 
\end{eqnarray*}

\noindent Hence, since $\mathcal{F}$ is closed under supersets, the desired property holds and we conclude that $\psi$ is a mapping of algebraic structures.

\medskip

\noindent Lastly, we show $\psi$ is injective: suppose to the contrary that $x \ne y$ in $A$ but $\psi_x = \psi_y$ in $\Tilde{A}$, i.e. the set $T = \{ \alpha | \psi_x(\alpha) = \psi_y(\alpha) \} \in \mathcal{F}$. Then by Lemma \ref{ultrafilter lemma}, $X-T = \{ \alpha | \psi_x(\alpha) \ne \psi_y(\alpha) \}$ is not in $\mathcal{F}$. However, 

\begin{eqnarray*}
X-T = \{ \alpha | \psi_x(\alpha) \ne \psi_y(\alpha) \} \supseteq \{ \alpha | \psi_x(\alpha) = x, \psi_y(\alpha) = y \} \supseteq X_{\{x\}} \cap X_{\{y\}} \in \mathcal{F}
\end{eqnarray*}

\noindent (This follows because $x \ne y$ and the sets $\{ x\}$ and $\{y\}$ are certainly finite subsets of $A$.) This is a contradiction--hence, $\psi$ is an injective map, and we can embed $A$ as a subsystem of the ultraproduct $\Tilde{A} = \prod A_\alpha / \mathcal{F}$.
\end{proof}

\bigskip

\textsc{University of California, San Diego, Department of Mathematics, 9500 Gilman Drive,
La Jolla, CA 92093-0112}

\medskip

\noindent \textit{Email address:} \texttt{jhennig@math.ucsd.edu}

\end{document}